\documentclass[11pt, reqno]{amsart}

\usepackage{amssymb}
\usepackage{graphicx}

\theoremstyle{plain}
\newtheorem{theorem}{Theorem}[section]
\newtheorem{lemma}[theorem]{Lemma}

\theoremstyle{definition}

\theoremstyle{remark}
\newtheorem{remark}{Remark}

\begin{document}

\title[Comparison of Convolution-Type Wave Equations ]{A Comparison of Solutions of Two Convolution-Type  Unidirectional Wave Equations}

\author[H. A. Erbay]{H. A. Erbay}

\address{Department of Natural and Mathematical Sciences, Faculty of Engineering, Ozyegin University, \\  Cekmekoy 34794, Istanbul, Turkey}
\email{husnuata.erbay@ozyegin.edu.tr}

\author[S. Erbay]{S. Erbay}

\address{Department of Natural and Mathematical Sciences, Faculty of Engineering, Ozyegin University, \\  Cekmekoy 34794, Istanbul, Turkey}
\email{saadet.erbay@ozyegin.edu.tr}

\author[A. Erkip]{A. Erkip}

\address{Faculty of Engineering and Natural Sciences, Sabanci University, \\ Tuzla 34956,  Istanbul,    Turkey}
\email{albert@sabanciuniv.edu}

\begin{abstract}
    In this work, we prove  a comparison result for  a general class of nonlinear  dispersive unidirectional wave equations. The dispersive nature of one-dimensional waves occurs because of  a convolution integral in space.      For two specific choices of  the kernel function, the Benjamin-Bona-Mahony equation   and the Rosenau equation that are particularly suitable to model  water waves and elastic waves, respectively, are  two members of the class. We first prove an energy estimate for the Cauchy problem of the nonlocal unidirectional wave equation. Then, for the same initial data,  we consider two  distinct  solutions corresponding to two different kernel functions. Our main result is that  the difference between the  solutions remains small in a suitable Sobolev norm if the  two  kernel functions  have  similar dispersive   characteristics  in the long-wave limit. As a sample case of this comparison result, we  provide  the approximations to the hyperbolic conservation law.
\end{abstract}



\maketitle

\setcounter{equation}{0}
\section{Introduction}\label{sec1}

In this paper, we establish a comparison result for solutions to the Cauchy problem associated to the one-dimensional nonlocal nonlinear wave equation
\begin{equation}
     u_{t} +\alpha \ast  \left(u+u^{p+1}\right)_{x}=0,  \label{cont}
\end{equation}%
under the assumption that kernel functions have  similar dispersive   characteristics  in the long-wave limit. Here  $u=u(x,t)$ is a real-valued function, $p$ is a positive integer, $\alpha(x)$ is a general kernel function  and   the symbol $\ast$ denotes convolution in the $x$-variable. The linear dispersion relation $\xi \mapsto \omega(\xi)=\xi \widehat{\alpha }(\xi)$,  where $\widehat{\alpha }$ represents the Fourier transform of $\alpha$, shows that  the dispersive nature of waves  is directly related to the kernel function.    Since we intend  to confine our interest to   long-wave solutions,   we rewrite  (\ref{cont}) in the form
\begin{equation}
     u_{t}+\alpha_{\delta }\ast ( u+ u^{p+1})_{x}=0   \label{para}
\end{equation}%
by utilizing  the transformation $(x, t) \rightarrow (x/\delta, t/\delta)$  with  small parameter $\delta >0$ and by introducing  the family of kernels as $\alpha _{\delta }(x)=\frac{1}{\delta }\alpha (\frac{x}{\delta })$.   It is worth noting that, as $\delta \rightarrow 0$, the kernels $\alpha_{\delta}$ will converge to the Dirac measure  in the distribution sense and (\ref{para}) will formally approach  the hyperbolic conservation law $u_{t}+ ( u+ u^{p+1})_{x}=0$. Indeed, (\ref{para}) is a dispersive regularization of  the hyperbolic conservation law  which was widely studied.

 The  nonlocal equation (\ref{para}) with a general kernel function was  given in \cite{Erbay2021}  to provide a numerical treatment of nonlinear unidirectional waves with a nonlocal dispersion relation. With particular kernel functions,   (\ref{para}) has been proposed as a model in a wide range of physical contexts.  The Benjamin-Bona-Mahony (BBM) equation \cite{Benjamin1972}
\begin{equation}
    u_{t}+u_{x}-\delta^{2}u_{xxt}+ (u^{p+1})_{x}=0  \label{ib}
\end{equation}%
is the most well-known member of the one parameter family (\ref{para}). It corresponds to the exponential kernel $\alpha_{\delta}(x)= \frac{1}{2\delta}e^{- \vert x\vert/\delta}$ and models unidirectional propagation of small amplitude long waves in shallow water.  Another well-known member of    (\ref{para})  is the Rosenau equation  \cite{Rosenau1988}
\begin{equation}
   u_{t}+u_{x}+\delta^{4}u_{xxxxt}+ (u^{p+1})_{x}=0 \label{clas}
\end{equation}%
which corresponds to the kernel function
\begin{equation}
    \alpha_{\delta}(x)= {1\over {2{\sqrt 2}\delta}}e^{-{\vert x\vert\over {\sqrt 2}\delta}}\Big( \cos\big({{\vert x\vert}\over {{\sqrt 2}\delta}}\big) + \sin \big({{\vert x\vert}\over {{\sqrt 2}\delta}}\big) \Big ) \label{ros-ker}
\end{equation}
and models  propagation of longitudinal waves on a one-dimensional dense chain of particles. It is worth to mention here that the BBM and Rosenau equations  might be viewed as degenerate cases of the  family (\ref{para}), because in both cases $\alpha_{\delta}$ is  the Green's function of a differential operator. However, in a "genuinely nonlocal" case, this is not the case and  (\ref{para}) cannot  be transformed into a partial differential equation. We underline that we will establish our comparison result for (\ref{para}) in which the kernel functions may or may not be   the Green's function of a differential operator and we pose minimal restrictions on the kernel.

The primary purpose of this work is to prove a comparison result of solutions to (\ref{para})  in the weak dispersive regime and also is to show that the behavior of solutions is determined by the dispersive character of the kernel in the long-wave limit rather than the shape of the kernel function. In a recent work \cite{Erbay2020}, a similar comparison result was given for the nonlocal bidirectional wave equations. Based mainly on an energy estimate with no loss of derivative, our present comparison result extends basically the notion of "kernel based comparison" introduced in \cite{Erbay2020}  to the nonlocal unidirectional wave equation  (\ref{para}).

We first start by considering two different kernel functions with the same dispersive nature in the long-wave limit. Then, for these two kernels, we consider the corresponding  solutions to the Cauchy problem with the same initial value. We basically prove that the difference of the two solutions remains small in a suitable norm. We refer the reader to \cite{Bona2005,Constantin2009,Duchene2015,Lannes2013,Lannes2012} and the references therein for a detailed discussion of similar comparison results of many different physical models.

The structure of the paper is as follows. Section \ref{sec2} is devoted to the proof of an energy estimate. In Section \ref{sec3} we start with  the moment conditions to be satisfied by the two different kernels and prove the main result that establishes an estimate on the difference between  the corresponding  solutions. In Section \ref{sec4},  we illustrate our comparison result through a particular case in which one of the kernels is the Dirac measure.

Throughout the paper, we will use the standard notation  for  Lebesgue and Sobolev spaces.  The $L^p$ ($1\leq p <\infty$)  norm of $u$ on $\mathbb{R}$ is represented by  $\Vert u\Vert_{L^p}$ and the notation $L^{p}=L^{p}(\mathbb{R})$ is used. To denote the inner product of $u$ and $v$ in $L^2$, the symbol $\langle u, v\rangle$ is used. The Fourier transform  of $u$ is defined as $\widehat u(\xi)=\int_\mathbb{R} u(x) e^{-i\xi x}dx$.  The notation  $H^{s}=H^s(\mathbb{R})$ is used to denote the $L^{2}$-based Sobolev space of order $s$ on $\mathbb{R}$,  with the norm $\Vert u\Vert_{H^{s}}=\big(\int_\mathbb{R} (1+\xi^2)^s \vert \widehat u(\xi)\vert^2 d\xi \big)^{1/2}$.  All the integrals in the paper will be over $\mathbb{R}$, so we will omit the limits of integration. $C$ is a generic positive constant.  We denote partial differentiations  by  $D_{x}$ etc.

\section{An energy estimate}\label{sec2}

The current section is devoted to the derivation of an energy estimate. Throughout this work we will assume that the kernel $ \alpha $ is an even  function in $L^{1}(\mathbb{R})$  with $\int \alpha (x)dx=1$, or more generally a finite Borel measure on $\mathbb{R}$ with $\int d\alpha =1$. For convenience, we will use the notation $Ku=\alpha \ast u$.  We note that the assumptions on $\alpha $ being an even $L^{1}(\mathbb{R})$ function, or more generally a finite Borel measure, imply that $K$ is a bounded and self-adjoint operator on $H^{s}$ for any $s$.

As $\ K$ also commutes with the derivative operator, for all $f\in H^{s}$, $s\geq 1$ we have
\begin{equation}
                \int (Kf)f_{x}dx=\int (Kf_{x})fdx=0.  \label{(Kf)f_x}
\end{equation}%
This identity follows from the self-adjointness of $K$ on $L^{2}(\mathbb{R})$ and integration by parts:
\begin{displaymath}
                \int (Kf)f_{x}dx=\int f(Kf_{x})dx=\int f(Kf)_{x}dx=-\int f_{x}(Kf)dx.
\end{displaymath}

In the rest of the work, we will use the notations $\Lambda^{s}=(1-D_{x}^{2})^{s/2}$ and $[\Lambda^{s},f] g=\Lambda^{s}(f g)-f\Lambda^{s}g$. Furthermore, to prove our estimates below and in the next section, we will need the following commutator estimates \cite{Kato1988}:
\begin{lemma}\label{lem2.1}
    Let  $s>0$. Then for all $f, g$ satisfying $f\in H^{s}$, $D_{x}f\in L^{\infty}$, $g\in H^{s-1}\cap L^{\infty}$,
    \begin{equation*}
        \big\Vert [ \Lambda^{s},f] g\big\Vert_{L^{2}}
        \leq C \big(\Vert D_{x}f\Vert_{L^{\infty}}\Vert g\Vert_{H^{s-1}}+\Vert f\Vert _{H^{s}}\Vert g\Vert_{L^{\infty }}\big).
    \end{equation*}%
    In particular, when $s>3/2$, due to the Sobolev embeddings $H^{s-1}\subset L^{\infty}$, for all $f, g\in H^{s}$
    \begin{equation*}
        \big\Vert [ \Lambda^{s},f] D_{x}g\big\Vert_{L^{2}}\leq C \Vert f\Vert_{H^{s}}\Vert g\Vert_{H^{s}}.
    \end{equation*}%
\end{lemma}

We will consider the related  linear equation
\begin{equation}
            u_{t}+K\big( (1+w)u_{x}\big) =F+G(u),\text{\ }x\in \mathbb{R},   \text{ \ \ }t>0, \label{lin}
\end{equation}
with the given functions $w(x,t)$, $F(x,t)$ and the linear map $G(u)$. Clearly, the original equation (\ref{cont})   is to be obtained from    (\ref{lin}) \ by setting $w=(p+1)u^{p}$ and $F=G=0$; hence (\ref{lin}) is a linearization of (\ref{cont}).

We define the $H^{s}$-energy functional by
\begin{equation}
        \mathcal{E}_{s}^{2}(t)=\frac{1}{2}\int \big(1+w(x,t)\big)\big(\Lambda^{s}u(x,t)\big)^{2}dx.  \label{energy}
\end{equation}%
Note that when $w$ is assumed to satisfy
\begin{equation}
        0<c_{1}\leq 1+w(x,t)\leq c_{2},  \label{hyperbolic}
\end{equation}%
where $c_{1}$ and $c_{2}$ are constants, $\mathcal{E}_{s}^{2}(t)$ is equivalent to the norm $\Vert u(t)\Vert _{H^{s}}^{2}$.
\begin{lemma}\label{lem2.2}
     Let $s>3/2$, $T>0$ and $F\in C\big([0,T],H^{s}\big)$, $w\in C\big([ 0,T],H^{s}\big)\cap C^{1}\big([0,T],H^{s-1}\big)$ with $0<c_{1}\leq 1+w(x,t)\leq c_{2}~$ for all $(x,t)\in \mathbb{R}\times \lbrack 0,T]$. Let $G$ be a linear map on $ C\big([0,T],H^{s}\big)$, satisfying $\Vert G(u(t))\Vert _{H^{s}}\leq C_{G}\Vert u(t)\Vert _{H^{s}}$ for all $t\in [0, T]$. Suppose $u\in C\big([0,T],H^{s}\big)$ satisfies (\ref{lin}) on $\mathbb{R}\times [0,T]$ with $u(x,0)=u_{0}(x)$. Then we have the estimate
    \begin{equation}
        \Vert u(t)\Vert _{H^{s}}\leq \Vert u_{0}\Vert _{H^{s}}\thinspace e^{At}+\frac{B}{A}(e^{At}-1)  \label{energy21}
    \end{equation}%
    for $0\leq t\leq T$, where
    \begin{eqnarray}
    \!\!\!\!\!\!\!\!\!
        A &=&\sup_{0\leq t\leq T}C\Big( \Vert w_{t}(t)\Vert _{L^{\infty }}+\big(1+\Vert w(t)\Vert _{L^{\infty}}\big) \big(C_{G}
            +\Vert \alpha \Vert_{L^{1}}\Vert w(t)\Vert_{H^{s}}\big)\Big),   \label{energy21A} \\
    \!\!\!\!\!\!\!\!\!
        B &=&\sup_{0\leq t\leq T}C\big(1+\Vert w(t)\Vert _{L^{\infty}}\big)\Vert F(t)\Vert _{H^{s}}.  \label{energy21B}
    \end{eqnarray}
\end{lemma}
\begin{proof}
    Differentiating (\ref{energy}) and using (\ref{lin}) we get
    \begin{eqnarray}
    \frac{d}{dt}\mathcal{E}_{s}^{2}(t)
        &=&\int \Big( \frac{1}{2}w_{t}(\Lambda ^{s}u)^{2}+(1+w)(\Lambda ^{s}u_{t})(\Lambda ^{s}u)\Big) dx  \nonumber \\
        &=&\int \Big( \frac{1}{2}w_{t}(\Lambda ^{s}u)^{2}+(1+w)(\Lambda^{s}F)(\Lambda ^{s}u)+(1+w)(\Lambda ^{s}G(u))(\Lambda ^{s}u) \nonumber \\
        &&~~~~~         -(1+w)(\Lambda^{s} K((1+w)u_{x})) (\Lambda ^{s}u)\Big) dx.   \label{epsqu}
    \end{eqnarray}%
    We handle the last term on the right-hand-side  separately as follows:
    \begin{eqnarray*}
        I(t)
        &=&\big\langle (1+w)\Lambda ^{s}K\big((1+w)u_{x}\big),\Lambda ^{s}u\big\rangle \\
        & =&\big\langle K\Lambda ^{s}\big((1+w)u_{x}\big),(1+w)\Lambda ^{s}u\big\rangle  \\
        &=&\big\langle \Lambda^{s}\big((1+w)u_{x}\big),K\big((1+w)\Lambda ^{s}u\big)\big\rangle  \\
        &=&\big\langle (1+w)\Lambda ^{s}u_{x}, K\big((1+w)\Lambda ^{s}u\big)\big\rangle
                +\big\langle \lbrack \Lambda ^{s}, 1+w \rbrack u_{x},K\big((1+w)\Lambda ^{s}u\big) \big\rangle\\
        &=&\big\langle ((1+w)\Lambda ^{s}u)_{x},K\big((1+w)\Lambda ^{s}u\big)\big\rangle
                -\big\langle w_{x}\Lambda ^{s}u,K\big((1+w)\Lambda ^{s}u\big)\big\rangle \\
         &&         +\big\langle \lbrack \Lambda^{s}, 1+w\rbrack u_{x},K\big((1+w)\Lambda ^{s}u\big)\big\rangle \\
        &=&-\big\langle w_{x}\Lambda ^{s}u,K\big((1+w)\Lambda ^{s}u\big)\big\rangle
                    +\big\langle \lbrack\Lambda ^{s},w\rbrack u_{x},K\big((1+w)\Lambda ^{s}u\big)\big\rangle,
    \end{eqnarray*}%
    where we have used   $[\Lambda^{s},1+w]u_{x}=[\Lambda ^{s},w]u_{x}+[\Lambda ^{s},1]u_{x} =[\Lambda ^{s},w]u_{x}$   and  (\ref{(Kf)f_x}) with $f=(1+w)\Lambda ^{s}u$. As $\Vert Kv\Vert _{L^{2}}\leq \Vert \alpha \Vert_{L^{1}}\Vert v\Vert _{L^{2}}$ and  $\Vert \Lambda ^{s}u\Vert _{L^{2}}=\Vert u\Vert_{H^{s}}$,
    \begin{eqnarray*}
            \vert I(t)\vert
            &\leq &\big\Vert w_{x}\Lambda ^{s}u\big\Vert_{L^{2}}\big\Vert K((1+w)\Lambda ^{s}u)\big\Vert _{L^{2}}
                    +\big\Vert[\Lambda ^{s},w]u_{x}\big\Vert _{L^{2}}\big\Vert K\big((1+w)\Lambda^{s}u\big)\big\Vert _{L^{2}} \\
            &\leq &\Vert \alpha \Vert _{L^{1}}\Big(\big \Vert w_{x}\Lambda^{s}u\big\Vert _{L^{2}
                        }\big\Vert (1+w)\Lambda ^{s}u\big\Vert_{L^{2}}
                        +\big\Vert [\Lambda ^{s},w]u_{x}\big\Vert _{L^{2}}\big\Vert(1+w)\Lambda ^{s}u\big\Vert _{L^{2}}\Big)  \\
            &\leq &\Vert \alpha \Vert _{L^{1}}\big( 1+\Vert w\Vert _{L^{\infty }}\big)
                      \Big( \Vert w_{x}\Vert_{L^{\infty }}\Vert \Lambda ^{s}u\Vert _{L^{2}}^{2}
                        +\big\Vert[\Lambda ^{s},w]u_{x}\big\Vert _{L^{2}}\Vert \Lambda ^{s}u\Vert_{L^{2}}\Big)  \\
            &\leq &\Vert \alpha \Vert _{L^{1}}\big( 1+\Vert w\Vert _{L^{\infty }}\big)
                      \Big( \Vert w_{x}\Vert_{L^{\infty }}\Vert u\Vert _{H^{s}}^{2}
                        +\big\Vert[\Lambda ^{s},w]u_{x}\big\Vert _{L^{2}}\Vert u\Vert_{H^{s}}\Big) .
    \end{eqnarray*}%
    Finally, since $s>3/2$, by the commutator estimate in Lemma \ref{lem2.1} we have
    \begin{displaymath}
            \big\Vert\lbrack \Lambda ^{s},w]u_{x}\big\Vert_{L^{2}}\leq C\Vert w\Vert_{H^{s}}\Vert u\Vert _{H^{s}}.
    \end{displaymath}
    So we get
    \begin{displaymath}
       \vert I(t)\vert \leq \Vert \alpha \Vert _{L^{1}}\big(1+\Vert w\Vert _{L^{\infty }}\big)
                    \big(  \Vert w_{x}\Vert _{L^{\infty }}+\Vert w\Vert _{H^{s}}\big)\Vert u\Vert _{H^{s}}^{2}.
    \end{displaymath}
        Then, from (\ref{epsqu}) we have
    \begin{eqnarray*}
        \frac{d}{dt}\mathcal{E}_{s}^{2}(t)
                &=&\int\Big( \frac{1}{2}w_{t}(\Lambda^{s}u)^{2}+(1+w)(\Lambda ^{s}F)(\Lambda ^{s}u)
                            +(1+w)(\Lambda ^{s}G(u))(\Lambda^{s}u)\Big) dx-I(t) \\
                &\leq &\frac{1}{2}\Vert w_{t}\Vert _{L^{\infty }}\Vert u\Vert _{H^{s}}^{2}
                            +(1+\Vert w\Vert _{L^{\infty}})\Big(\Vert F\Vert _{H^{s}}\Vert u\Vert_{H^{s}}
                            +\Vert G(u)\Vert_{H^{s}}\Vert u\Vert _{H^{s}} \Big) \\
               &&             +\Vert \alpha \Vert_{L^{1}}( 1+\Vert w\Vert _{L^{\infty }})
                            \big(\Vert w_{x}\Vert _{L^{\infty }}+\Vert w\Vert _{H^{s}}\big)\Vert u\Vert _{H^{s}}^{2}.
    \end{eqnarray*}%
     Since $\Vert G(u(t))\Vert _{H^{s}}\leq C_{G}\Vert u(t)\Vert _{H^{s}}$, we have
    \begin{eqnarray}
     \frac{d}{dt}\mathcal{E}_{s}^{2}(t)
            &\leq&  \Big(\frac{1}{2}\Vert w_{t}\Vert_{L^{\infty }} +( 1+\Vert w\Vert _{L^{\infty }})\Big(C_{G}
                +\Vert \alpha \Vert_{L^{1}}\big(\Vert w_{x}\Vert _{L^{\infty }}  +\Vert w\Vert _{H^{s}}\big)\Big)\Big) \Vert u\Vert_{H^{s}}^{2}  \nonumber  \\
           &&+ ( 1+\Vert w\Vert _{L^{\infty }})\Vert F\Vert _{H^{s}} \Vert u\Vert_{H^{s}} .
   \end{eqnarray}
   But  $\Vert u(t)\Vert _{H^{s}}^{2}\leq \frac{2}{c_{1}}\mathcal{E}_{s}^{2}(t)$ due to (\ref{energy})-(\ref{hyperbolic}), so
      \begin{eqnarray}
       \frac{d}{dt}\mathcal{E}_{s}^{2}(t)
            &\leq&  C\Big(\Vert w_{t}\Vert_{L^{\infty }} +( 1+\Vert w\Vert _{L^{\infty }})\Big(C_{G}
                +\Vert \alpha \Vert_{L^{1}}\big(\Vert w_{x}\Vert _{L^{\infty }}  +\Vert w\Vert _{H^{s}}\big)\Big)\Big)\mathcal{E}_{s}^{2}(t)  \nonumber \\
           &&    +C( 1+\Vert w\Vert _{L^{\infty }})      \Vert F\Vert_{H^{s}}\mathcal{E}_{s}(t), \nonumber \\
           \frac{d}{dt}\mathcal{E}_{s}(t)
                &\leq &C\Big(\Vert w_{t}\Vert_{L^{\infty }}+( 1+\Vert w\Vert _{L^{\infty }})\Big(C_{G}
                    +\Vert \alpha \Vert_{L^{1}}\big(\Vert w_{x}\Vert _{L^{\infty }} +\Vert w\Vert _{H^{s}}\big)\Big)\Big)\mathcal{E}_{s}(t) \nonumber  \\
               &&    +C( 1  +\Vert w\Vert _{L^{\infty }}) \Vert F\Vert _{H^{s}}, \nonumber \\
                &\leq &A\mathcal{E}_{s}(t)+B, \label{grr}
    \end{eqnarray}%
    where
    \begin{eqnarray*}
            A &=&\sup_{0\leq t\leq T}C\Big( \Vert w_{t}(t)\Vert _{L^{\infty }}+(1+\Vert w(t)\Vert _{L^{\infty}}) \big(C_{G}
            +\Vert \alpha \Vert_{L^{1}}\Vert w(t)\Vert_{H^{s}}\big)\Big) , \\
            B &=&\sup_{0\leq t\leq T}C( 1+\Vert w(t)\Vert _{L^{\infty}}) \Vert F(t)\Vert _{H^{s}}.
    \end{eqnarray*}%
   We note that when $\alpha$ is a measure, the quantity $\Vert \alpha \Vert_{L^{1}}$ should be replaced by $\vert \alpha \vert (\mathbb{R})$.  An application of Gronwall's inequality to (\ref{grr}) then yields
    \begin{displaymath}
            \mathcal{E}_{s}(t)\leq \mathcal{E}_{s}(0)e^{At}+\frac{B}{A}(e^{At}-1).
    \end{displaymath}
    Since $\mathcal{E}_{s}(t)\approx \Vert u(t)\Vert _{H^{s}},$ \ this completes the proof of (\ref{energy21}).
\end{proof}

The local well-posedness of the Cauchy problem for (\ref{cont}) follows from the standard hyperbolic approach \cite{Taylor2011}, where the main tool is the energy estimate of Lemma \ref{lem2.2}. To be precise we have:
\begin{theorem}\label{theo2.3}
   Suppose $s>3/2$,  $u_{0}\in H^{s}$ with sufficiently small $\Vert u_{0} \Vert _{H^{s}}$. Then there exists $T>0$; so that   the Cauchy problem for (\ref{cont}) with initial data $u_{0}$ has a unique solution $u\in C\big([0, T], H^{s}\big)\cap C^{1}\big([0, T], H^{s-1}\big)$. Moreover, the existence time $T$ is of the form $T=T(\Vert u_{0} \Vert _{H^{s}}, \Vert \alpha \Vert_{L^{1}})$; in other words it depends only on $\Vert u_{0} \Vert _{H^{s}}$ and $\Vert \alpha \Vert_{L^{1}}$.
\end{theorem}
The idea of the proof is as follows.  Roughly speaking, the nonlinear problem  (\ref{cont}) can be considered as the linearized problem (\ref{lin}) with $w=(p+1)u^{p}$ and $F=G=0$. The condition $0<c_{1}\leq 1+(p+1)u^{p}\leq c_{2}$ is then achieved at $t=0$ by the smallness assumption of $\Vert u_{0} \Vert _{H^{s}}$ and carried on by continuity for small times. From Lemma \ref{lem2.2} we see that this can be achieved for small $At$; in other words the existence time should satisfy $T=\mathcal{ O}(A^{-1})$. Again when $t=0$, $A$ depends on $\Vert u_{0} \Vert _{H^{s}}$ and $\Vert \alpha \Vert_{L^{1}}$; by continuity this can be carried over small times using the energy estimate, yielding  $T=T(\Vert u_{0} \Vert _{H^{s}}, \Vert \alpha \Vert_{L^{1}})$.
\begin{remark}\label{rem2.4}
    Consider the parameter dependent family of  (\ref{para}). Since  $\Vert \alpha_{\delta} \Vert_{L^{1}}=\Vert \alpha \Vert_{L^{1}}$, the solutions $u^{\delta}$   with the same initial data $u_{0}$ will have a uniform existence time $T$ independent of  $\delta$. Moreover, again by Lemma \ref{lem2.2},  $\Vert u^{\delta}(t) \Vert _{H^{s}}$ will be uniformly bounded in $\delta$ and $t\in [0, T]$.
\end{remark}
\begin{remark}\label{rem2.5}
If the above theorem is applied to  the  equation  $u_{t} +\alpha\ast (u+\epsilon^{p}u^{p+1})_{x}=0$ with small nonlinearity  $\epsilon^{p}u^{p+1}$, Lemma \ref{lem2.2} with $F=G=0$ gives $A=\mathcal{O}(\epsilon^{p})$. In other words, the existence time turns out to be $T^{\epsilon}=\mathcal{O}(\frac{1}{\epsilon^{p}})$.
\end{remark}

\section{Comparison of Solutions}\label{sec3}

In this section, we restrict ourselves to the kernel functions  that have the same dispersive nature in the long-wave limit  and  we  prove our main  theorem, which states that the corresponding solutions with the same initial data  are “close” in a sense which will be made precise below.

Suppose that $\alpha^{(1)}$ and $\alpha ^{(2)}$ are two different kernel  functions  satisfying the following  conditions for some $k\geq 1$:
\begin{itemize}
    \item[(C1)] $\alpha^{(1)}$ and $\alpha ^{(2)}$ have the same first $(2k-1)$-order moments, namely
        \begin{equation}
            \int x^{j}\alpha ^{(1)}(x)dx=\int x^{j}\alpha ^{(2)}(x)dx~~~\mbox{ for }~~0\leq j<2k-1,  \label{moments-real}
        \end{equation}
    \item[(C2)]  $x^{2k}\alpha ^{(i)}(x)\in L^{1}\left( \mathbb{R}\right) $ for $i=1,2$.
\end{itemize}
Notice that if $\alpha =\mu $ is a finite measure, we should replace   the moment integral  by $\int x^{j}d\mu$. We now consider the Cauchy problem for  (\ref{para}) with initial data $u_{0}$. Let $u_{1}^{\delta}$ and $u_{2}^{\delta}$ be solutions of the Cauchy problem, corresponding to the kernels   $\alpha^{(1)}_{\delta}$ and $\alpha ^{(2)}_{\delta}$, respectively.  In the following main result of this paper, we prove that the solutions  $u_{1}^{\delta}$ and $u_{2}^{\delta}$   are "close" to each other.   The idea of the  proof is similar to that  already sketched in Theorem 3.2 of \cite{Erbay2020}.
\begin{theorem}\label{theo3.1}
        Let $\alpha^{(1)}$ and $\alpha^{(2)}$ be two kernels satisfying the conditions (C1) and (C2)  for  some $k\geq 1$. Let  $s>3/2$ and $u_{0}\in H^{s+2k+1}$ with sufficiently small $\Vert u_{0} \Vert _{H^{s+2k+1}}$. Then there are some constants  $C$ and $T>0$  independent of  $\delta $ so that the solutions $u_{i}^{\delta }$ of the Cauchy problems
        \begin{eqnarray}
                && (u_{i})_{t} +\alpha _{\delta }^{\left( i\right) }\ast (u_{i}+u_{i}^{p+1})_{x}=0,~~~~~x\in \mathbb{R}, ~~~~t>0,  \label{main-a} \\
                && u_{i}(x,0) =u_{0} (x),~~~~x\in \mathbb{R},  \label{main-b}
        \end{eqnarray}%
        for $\ i=1,2$ are defined for all $~t\in \left[ 0,T\right] $ and satisfy
        \begin{equation}
                \Vert u_{1}^{\delta}(t) -u_{2}^{\delta}(t) \Vert _{H^{s}}   \leq C\delta^{2k}t\mbox{ \  \  for all \  }t\leq T.   \label{mainresult}
    \end{equation}
\end{theorem}
\begin{proof}
       The proof will be split into several steps.
        \begin{description}
        \item[Step 1] Let  $u_{0} \in H^{s+2k+1}$. By Theorem \ref{theo2.3} and Remark \ref{rem2.4} applied with $s$ replaced by $s+2k+1$, solutions $u^{\delta}_{i}$ exist in $C\big([0, T], H^{s+2k+1}\big)\cap C^{1}\big([0, T], H^{s+2k}\big)$, and  we have the uniform existence time and the uniform bound
            \begin{eqnarray*}
                    \left\Vert u^{\delta}_{i}(t)\right\Vert _{H^{s+2k+1}}      & \leq C \mbox{ \ \ for all\ \ } \delta \text{ \ \ and\ \ } t\leq T,
            \end{eqnarray*}
            for both families of solutions.
        \item[Step 2]  From now on, we will drop the superscript $\delta $ to simplify the presentation. We will use the following form of (\ref{main-a})-(\ref{main-b})
                       \begin{equation}
                    ( u_{i})_{t} +K_{\delta }^{(i)}\big(u_{i}+u_{i}^{p+1}\big)_{x}=0,~~~~ u_{i}(x,0) =u_{0} (x)
            \end{equation}%
            for $\ i=1,2$. Let $r$ denote the difference   between the solutions $u_{1}$ and $u_{2}$, i.e.  $r=u_{1}-u_{2}$. Then $r$ satisfies
            \begin{equation}
                    r_{t} +K_{\delta }^{(1)}\big((1+w)r_{x}\big)=F+G(r),        \quad \quad       r(x,0)=0,  \label{rho-sys}
            \end{equation}%
            where
            \begin{eqnarray}
                    && F =\big(K_{\delta }^{(2)}-K_{\delta }^{(1)}\big)\big((u_{2})_{x}+(p+1)u_{1}^{p}(u_{2})_{x}\big),   \label{rF}\\
                    && G(r) =G(u_{1}-u_{2})=(p+1)K_{\delta}^{(2)}\big((u_{2}^{p}-u_{1}^{p})(u_{2})_{x}\big),   \label{rG}  \\
                    && w=(p+1)u_{1}^{p}.  \label{Rw}
            \end{eqnarray}%
        \item[Step 3]
           Under the conditions (C1) and (C2) described above,  the Fourier transforms of the kernels satisfy $\widehat{\alpha ^{(i)}}\in C^{2k}$ for $i=1,2$ and
            \begin{equation}
                    \frac{d^{j}}{d\xi^{j}}\left( \widehat{\alpha ^{(2)}}(\xi )-\widehat{\alpha^{(1)}}(\xi )\right) \Big|_{\xi =0}=0~~\mbox{ for }~~0\leq j<2k-1. \label{moments-fourier}
            \end{equation}
            Then  $\widehat{\alpha ^{(2)}}(\xi )-\widehat{\alpha ^{(1)}}(\xi )=\mathcal{O}\left( \xi ^{2k}\right)$ near the origin and  we have
            \begin{displaymath}
                    \widehat{\alpha ^{(2)}}(\xi )-\widehat{\alpha ^{(1)}}(\xi )=\xi ^{2k}m(\xi )
            \end{displaymath}%
            for some continuous function $m$. Since both $\widehat{\alpha ^{(1)}}(\xi )$ and $\widehat{\alpha^{(2)}}(\xi )$ are bounded, $m(\xi )$ is also bounded.   With $\widehat{\alpha_{\delta}^{(i)}}(\xi )= \widehat{\alpha^{(i)}}(\delta\xi )$ for $i=1,2$,  the associated  operators $K_{\delta}^{(1)}$  and  $K_{\delta}^{(2)}$    will satisfy
            \begin{displaymath}
                    K_{\delta }^{(2)}=K_{\delta }^{(1)}+\left( -1\right) ^{k}\delta^{2k}D_{x}^{2k}M_{\delta },
            \end{displaymath}%
            where $M_{\delta }$ is the operator with symbol $m(\delta \xi )$. By the boundedness of  $m$, it follows that
            \begin{equation}
                    \left\Vert \big(K_{\delta }^{(2)}-K_{\delta }^{(1)}\big)u\right\Vert_{H^{s}}
                     = \delta^{2k}\left\Vert  D_{x}^{2k}M_{\delta }u\right\Vert _{H^{s}}
                    \leq C\delta^{2k}\left\Vert u\right\Vert _{H^{s+2k}},  \label{K2K1est}
            \end{equation}%
            with the constant $C$ independent of $\delta $.
        \item[Step 4]
            We now estimate the terms  $w$, $F$ and  $G$  in (\ref{rho-sys}).
            From the definition $w=(p+1)u_{1}^{p}$, we get
           \begin{eqnarray}
                \left\Vert w(t)\right\Vert _{H^{s}}    &\leq& C\left\Vert u_{1}(t)\right\Vert _{H^{s}}^{p}          \leq C, \label{estW}  \\
                \left\Vert w_{t}(t)\right\Vert _{H^{s-1}}    &=&\left\Vert p(p+1) u_{1}^{p-1}(t)(u_{1})_{t}(t)\right\Vert _{H^{s-1}}  \nonumber \\
                                                                                             &\leq&   C\left\Vert u_{1}(t)\right\Vert _{H^{s-1}}^{p-1}   \left \Vert (u_{1})_{t}(t)\right\Vert _{H^{s-1}}
                                                                                            \leq C.  \label{estWt}
            \end{eqnarray}
            We have
            \begin{eqnarray}
                    \left\Vert F(t)\right\Vert _{H^{s}}
                                &=&  \left\Vert\big(K_{\delta }^{(2)}-K_{\delta }^{(1)}\big)\big((u_{2})_{x}+(p+1)u_{1}^{p}(u_{2})_{x}\big) \right\Vert _{H^{s}} \nonumber \\
                                &\leq & C\delta^{2k}  \big\Vert (u_{2})_{x}+(p+1)u_{1}^{p}(u_{2})_{x} \big\Vert _{H^{s+2k}} \nonumber \\
                               &\leq & C\delta^{2k}   \big( \left\Vert  (u_{2})_{x}(t)\right\Vert _{H^{s+2k}}+\left\Vert u_{1}(t) \right\Vert _{H^{s+2k}}^{p}\left\Vert  (u_{2})_{x}(t)\right\Vert _{H^{s+2k}}\big)  \nonumber \\
                                 &\leq & C\delta^{2k}\big(\left\Vert  u_{2}(t)\right\Vert _{H^{s+2k+1}}+
                   \left\Vert u_{1}(t) \right\Vert _{H^{s+2k}}^{p}\Vert u_{2}(t)\Vert _{H^{s+2k+1}}\big)\nonumber \\
                     &\leq &      C\delta ^{2k} \label{estF1}
            \end{eqnarray}
            where $\Vert u_{i}(t)\Vert _{H^{s+2k+1}}$ for $i=1,2$ are bounded.     Similarly,
            \begin{eqnarray}
                    \left\Vert G(r)(t)\right\Vert _{H^{s}}
                    &=&\left\Vert (p+1)K_{\delta}^{(2)}\Big(\big(u_{2}^{p}-u_{1}^{p}\big)\big(u_{2}\big)_{x}\Big)(t)\right\Vert_{H^{s}} \nonumber \\
                    &\leq &C  \Vert \alpha_{\delta}^{(2)} \Vert_{L^{1}}  \Vert r(t)\Vert _{H^{s}} \Vert u_{2}(t)\Vert_{H^{s+1}}  \leq C_{G}\Vert r(t)\Vert _{H^{s}}   \label{estF2}
            \end{eqnarray}
             where we have used the facts that $\Vert u_{i}(t)\Vert_{H^{s}}\leq C$   and
            \begin{displaymath}
                    u_{2}^{p}-u_{1}^{p}=(u_{2}-u_{1})(u_{2}^{p-1}+u_{2}^{p-2}u_{1}+\cdots +u_{2}u_{1}^{p-2}+u_{1}^{p-1}).
             \end{displaymath}
        \item[Step 5]
            We now apply Lemma \ref{lem2.2} to the solution $r$ of the initial-value problem (\ref{rho-sys}). Since $\Vert r(0)\Vert _{H^{s}}=0$, by (\ref{energy21})  we have the estimate
             \begin{equation}
                    \Vert r(t)\Vert _{H^{s}}\leq \frac{B}{A}(e^{At}-1)   ~~~ \text{for all}~t\leq T,   \label{restimate}
            \end{equation}
            where $A$ and $B$ are given by  (\ref{energy21A}) (with $\alpha$ replaced by $\alpha_{\delta}^{(1)}$) and  (\ref{energy21B}), respectively.
            Using the estimates on  $\left\Vert w\right\Vert _{H^{s}}$,  $\left\Vert w_{t}\right\Vert_{H^{s-1}}$,   $\left\Vert F \right\Vert _{H^{s}}$  and $\left\Vert G\right\Vert _{H^{s}}$ given by (\ref{estW}),    (\ref{estWt}),   (\ref{estF1}) and (\ref{estF2}), respectively,  in (\ref{energy21A}) and  (\ref{energy21B}), we get from   (\ref{restimate}) that
                 \begin{displaymath}
                     \Vert r(t)\Vert _{H^{s}} =\Vert u_{1}(t) -u_{2}(t) \Vert _{H^{s}}\leq  C\delta^{2k} t   ~~~ \text{for all}~t\leq T.
            \end{displaymath}
            This completes the proof.
\end{description} \end{proof}

\section{Convergence  to the hyperbolic conservation law}\label{sec4}

In this section we  provide a  justification of the convergence of the nonlocal (dispersive)  models  to the hyperbolic conservation law in the limit of vanishing nonlocality, that is, in the zero-dispersion limit.

When the kernel $\alpha$ is taken as the Dirac measure, (\ref{para}) becomes the hyperbolic conservation law
\begin{equation}
     u_{t}+ u_{x}+ (u^{p+1})_{x}=0.   \label{hopf}
\end{equation}%
This  equation  is also called the inviscid Burgers equation or the Hopf equation and  arises in  various fields such as  fluid dynamics, traffic flow, acoustics.  The nonlocal equation (\ref{para}) with a general kernel function may be considered as a  dispersive regularization of (\ref{hopf}).  From now on, we take one of the two kernel functions considered in the previous section as the Dirac measure.   One question for (\ref{hopf})  is whether  solutions of the nonlocal wave equation (\ref{para}) converge to the solution of the non-dispersive equation (\ref{hopf}) when the dispersion parameter $\delta$ tends to zero.  Under rather general assumptions this question has been answered in the following comparison result which  is an immediate consequence of Theorem \ref{theo3.1}.  So, for sufficiently smooth initial conditions, the  solutions of (\ref{para}) and (\ref{hopf}) aproximate  each other with an approximation error of order  $\delta^{2k}$  as   the kernel $\alpha$ approaches the Dirac distribution.
\begin{theorem}\label{theo4.1}
     Suppose that the zeroth-order moment of $\alpha$ is 1 and that the first $(2k-1)$-order moments of $\alpha$ for  some $k\geq 1$ are 0. Let  $s>3/2$ and $u_{0}\in H^{s+2k+1}$   with sufficiently small $\Vert u_{0} \Vert _{H^{s+2k+1}}$.  Suppose also that  $u^{\delta}$ and  $u$ satisfy (\ref{para})  and (\ref{hopf}), respectively,  with the initial data $u^{\delta}(x,0)= u(x,0) =u_{0} (x)$.  Then there are some constants  $C$ and $T>0$  independent of  $\delta $ so that
    \begin{equation}
        \Vert u^{\delta }\left( t\right) -u\left( t\right) \Vert _{H^{s}}   \leq C\delta^{2k}t       \mbox{ \  \  for all \  }t\leq T.    \label{mainresult}
    \end{equation}
\end{theorem}

We will  consider some  examples of the kernel functions for which   one would expect convergence of the nonlocal (dispersive) solution of (\ref{para}) to the solution of  (\ref{hopf}). As it is expected, the smaller the dispersion parameter, the closer the nonlocal (dispersive) solution of (\ref{para}) is to the  solution of (\ref{hopf}).   We start with considering the kernel function $\alpha_{\delta}(x)$ whose Fourier transform is
\begin{displaymath}
            \widehat{\alpha_{\delta}}(\xi)=\big(1+\delta^{2k} \xi^{2k}\big)^{-1}
\end{displaymath}
with  $k\geq 1$.  The dispersive wave equation corresponding to $\alpha_{\delta}(x)$ is
\begin{equation}
             L u_{t}+ u_{x}+ (u^{p+1})_{x}=0, ~~~~~L=1-(-1)^{k-1}\delta^{2k}D_{x}^{2k}. \label{highBBM}
\end{equation}
As both of the Fourier transforms of $\alpha_{\delta}(x)$ and the Dirac measure  have exactly the same Taylor expansion (around the origin) to order $2k-1$,  the solutions of (\ref{highBBM}) and (\ref{hopf}) approximate each other with an error of order $\delta^{2k}$  by Theorem \ref{theo4.1}.  We recall that, if $\big(\widehat{\alpha_{\delta}}(\xi)\big)^{-1}$ is a polynomial in $\xi$, then  (\ref{para})  takes the form of a differential equation rather than an integro-differential equation.  When  $k=1$,  $\alpha_{\delta}(x)$ is the exponential kernel $\alpha_{\delta}(x)= \frac{1}{2\delta}e^{- \vert x\vert/\delta}$ and   (\ref{highBBM}) reduces to the BBM equation  (\ref{ib}).  Similarly, when $k=2$, $\alpha_{\delta}(x)$ is given by (\ref{ros-ker}) and   (\ref{highBBM}) becomes the Rosenau equation  (\ref{clas}). So the  approximation error in Theorem \ref{theo4.1} will be  of the order of  $\delta^{2}$ and $\delta^{4}$  for the BBM and Rosenau approximations of (\ref{hopf}), respectively.

More generally, we can consider the fractional BBM-type equation
\begin{equation}
    u_{t}+u_{x}+\delta^{2\gamma}(-D_{x}^{2})^{\gamma}u_{t}+ (u^{p+1})_{x}=0, \label{Fib}
\end{equation}%
where  $\gamma >0$ is not necessarily an integer. This equation is of the form (\ref{para}) with the kernel defined by
\begin{displaymath}
            \widehat{\alpha_{\delta}}(\xi)=\big(1+\delta^{2\gamma} \vert\xi\vert^{2\gamma}\big)^{-1}.
\end{displaymath}
Since  $\widehat{\alpha_{\delta}}(\xi)=1-\delta^{2\gamma} \vert\xi\vert^{2\gamma}+\mathcal{O} \big(\vert\xi\vert^{2\gamma+1}\big)$ near the origin, the approximation  error between solutions of the fractional equation (\ref{Fib}) and the hyperbolic conservation law (\ref{hopf})  is order   $\mathcal{O}\big(\delta^{2\gamma}t\big)$ by Theorem \ref{theo4.1}.

We now consider  the rectangular kernel  defined by
\begin{equation*}
           \alpha_{\delta}(x)=\frac{1}{\delta}\left\{
        \begin{array}{cc}
        1 & \text{ for }~|x|\leq \frac{\delta}{2}, \\~\\
        0 & \text{ for }~|x|>\frac{\delta}{2},
        \end{array}%
        \right.
\end{equation*}%
for which   $ \widehat{\alpha_{\delta}}(\xi )=\frac{2}{\delta\xi}\sin ( \frac{\delta\xi }{2})$.   Then (\ref{para}) becomes the differential-difference equation
\begin{equation}
    u_{t}+\nabla _{\delta }^{d}\big(u+u^{p+1}\big)=0  \label{lattice3}
\end{equation}%
 if we use  the difference operator
\begin{equation}
    \nabla _{\delta }^{d}v(x,t)=\frac{1}{\delta}\Big(v(x+\frac{\delta}{2},t)-v(x-\frac{\delta}{2},t)\Big).       \label{lap3}
\end{equation}
 As a member of the class  (\ref{para}), the differential-difference equation (\ref{lattice3}) is the standard second-order (a three-point stencil) central finite-difference scheme in space for  the hyperbolic conservation law (\ref{hopf}). Since   $\widehat{\alpha_{\delta}}(\xi )=1-\delta^{2}\xi^{2}/24+\cdots$ about the origin, Theorem \ref{theo4.1} tells us that  the difference of the solutions of (\ref{lattice3}) and (\ref{hopf}) with the same initial data  tends to zero as $\delta \rightarrow 0$ with a rate of $\delta^{2}$.

As an another example we now consider  the kernel  defined by
\begin{equation*}
         \alpha_{\delta}(x)=\frac{1}{6\delta}\left\{
        \begin{array}{rcc}
         -1 & \text{ for }~&  -\delta \leq x < -\frac{\delta}{2}, \\~\\
          7 & \text{ for }~&  -\frac{\delta}{2} \leq x < \frac{\delta}{2}, \\~\\
         -1 & \text{ for }~&\frac{\delta}{2} \leq x <\delta,   \\~\\
         0   &    ~               &  \text{otherwise}
        \end{array}%
        \right.
\end{equation*}%
for which   $ \widehat{\alpha_{\delta}}(\xi )=\frac{8}{3\delta\xi}\sin ( \frac{\delta\xi }{2})-\frac{1}{3\delta\xi}\sin (\delta\xi)$.   Then (\ref{para}) becomes the differential-difference equation
\begin{equation}
    u_{t}+\widetilde{\nabla}_{\delta }^{d}\big(u+u^{p+1}\big)=0  \label{lattice5}
\end{equation}%
 if we use  the difference operator
\begin{equation}
   \widetilde{\nabla}_{\delta }^{d}v(x,t)=\frac{1}{6\delta}\Big(-v(x+\delta,t)+8v(x+\frac{\delta}{2},t)-8v(x-\frac{\delta}{2},t)+v(x-\delta,t)\Big).     \label{lap5}
\end{equation}
 As a member of the class  (\ref{para}), the differential-difference equation (\ref{lattice5})  is the standard fourth-order (a five-point stencil) central finite-difference scheme in space for the hyperbolic conservation law (\ref{hopf}). Since   $\widehat{\alpha_{\delta}}(\xi )=1-\delta^{4}\xi^{4}/480+\cdots$ about the origin, by Theorem \ref{theo4.1} we know that  the difference of the solutions of (\ref{lattice5}) and (\ref{hopf}) with the same initial data  tends to zero at a quartic rate as $\delta \rightarrow 0$.


\end{document}